\documentclass[a4paper,10pt,reqno]{amsart}

\usepackage[utf8]{inputenc}
\usepackage[english]{babel}
\usepackage{lmodern}
\usepackage[T1]{fontenc}
\usepackage{amssymb, amsmath, amsthm}
\usepackage[Symbol]{upgreek}
\usepackage{xy}
\xyoption{all}
\def\umono{\ar@{_{(}->}[u]}
\def\uumono{\ar@{_{(}->}[uu]}

\def\lmono{\ar@{_{(}->}[l]}
\def\llmono{\ar@{_{(}->}[ll]}

\DeclareFontFamily{U}{MnSymbolC}{}
\DeclareFontShape{U}{MnSymbolC}{m}{n}{
   <-6>  MnSymbolC5
  <6-7>  MnSymbolC6
  <7-8>  MnSymbolC7
  <8-9>  MnSymbolC8
  <9-10> MnSymbolC9
 <10-12> MnSymbolC10
 <12->   MnSymbolC12}{}
\DeclareFontShape{U}{MnSymbolC}{b}{n}{
   <-6>  MnSymbolC-Bold5
  <6-7>  MnSymbolC-Bold6
  <7-8>  MnSymbolC-Bold7
  <8-9>  MnSymbolC-Bold8
  <9-10> MnSymbolC-Bold9
 <10-12> MnSymbolC-Bold10
 <12->   MnSymbolC-Bold12}{}
\DeclareSymbolFont{MnSyC}{U}{MnSymbolC}{m}{n}
\SetSymbolFont{MnSyC}{bold}{U}{MnSymbolC}{b}{n}
\DeclareMathSymbol{\lefthalfcup}{\mathord}{MnSyC}{183}
\DeclareMathSymbol{\righthalfcup}{\mathord}{MnSyC}{184}
\DeclareMathSymbol{\lefthalfcap}{\mathord}{MnSyC}{185}
\DeclareMathSymbol{\righthalfcap}{\mathord}{MnSyC}{186}
\DeclareMathSymbol{\medsquare}{\mathord}{MnSyC}{106}

\newcommand*{\set}[2]{\left\{#1\;\middle\vert\;#2\right\}} 

\newcommand*{\defas}{\mathrel{\mathrel{\mathop:}=}}

\def\letters{a,b,c,d,e,f,g,h,i,j,k,l,m,n,o,p,q,r,s,t,u,v,w,x,y,z}
\def\Letters{A,B,C,D,E,F,G,H,I,J,K,L,M,N,O,P,Q,R,S,T,U,V,W,X,Y,Z}
\makeatletter
\@for \@l:=\Letters \do{%
 \expandafter\edef\csname\@l bb\endcsname{\noexpand\ensuremath{\noexpand\mathbb{\@l}}}%
 \expandafter\edef\csname\@l bf\endcsname{{\noexpand\bf \@l}}%
 \expandafter\edef\csname\@l cal\endcsname{\noexpand\ensuremath{\noexpand\mathcal{\@l}}}%
 \expandafter\edef\csname\@l eu\endcsname{\noexpand\ensuremath{\noexpand\EuScript{\@l}}}%
 \expandafter\edef\csname\@l frak\endcsname{\noexpand\ensuremath{\noexpand\mathfrak{\@l}}}%
 \expandafter\edef\csname\@l rm\endcsname{{\noexpand\rm \@l}}%
 \expandafter\edef\csname\@l scr\endcsname{\noexpand\ensuremath{\noexpand\mathscr{\@l}}}%
}
\@for \@l:=\letters \do{%
 \expandafter\edef\csname\@l bf\endcsname{{\noexpand\bf \@l}}%
 \expandafter\edef\csname\@l frak\endcsname{\noexpand\ensuremath{\noexpand\mathfrak{\@l}}}%
 \expandafter\edef\csname\@l rm\endcsname{{\noexpand\rm \@l}}%
}
\makeatother

\def\OmegaB{{\Upomega}}

\newcommand*{\hocolim}{\operatornamewithlimits{hocolim}}

\newcommand*{\Fib}{\mathop{\rm Fib}\nolimits}
\newcommand*{\Fibset}{\mathop{\mathcal Fib}\nolimits}

\newcommand*{\Ker}{\mathop{\rm Ker}\nolimits}

\renewcommand*{\coprod}{\amalg}
\let\dummycoprod\coprod
\renewcommand*{\coprod}{\mathbin{\dummycoprod}}

\theoremstyle{plain}
\newtheorem{thm}{Theorem}

\newtheorem{theorem}{Theorem}[section]
\newtheorem{proposition}[theorem]{Proposition}
\newtheorem{corollary}[theorem]{Corollary}

\newtheorem{lemma}[theorem]{Lemma}

\theoremstyle{definition}
\newtheorem{definition}[theorem]{Definition}
\newtheorem{example}[theorem]{Example}

\newtheorem{point}[theorem]{}

\title[Homotopy excision and cellularity]{Homotopy excision and cellularity}

\author[W. Chach\'olski]{Wojciech Chach\'olski}
\author[J. Scherer]{J\'er\^ome Scherer}
\author[Kay Werndli]{Kay Werndli}

\address{Department of Mathematics\\ KTH Stockholm
Lindstedtsv\" agen 25 \\ 10044 Stockholm \\ Sweden}
\email{wojtek@math.kth.se}

\address{Department of Mathematics\\ EPFL Lausanne \\
Station 8 \\1015 Lausanne \\ Switzerland }
\email{jerome.scherer@epfl.ch, kay.werndli@epfl.ch}

\thanks{W. Chach\'olski was partially supported by G\"oran Gustafsson Stiftelse and VR grants. J. Scherer was
partially supported by MTM201020622, and UNAB104E378 ``Una manera de hacer Europa''. K. Werndli was
supported by Swiss NSF doctoral grant 200020 149118.}


\begin{document}


\begin{abstract}
Consider a push-out diagram of spaces $C \leftarrow A \rightarrow B$, construct the
homotopy push-out, and then the homotopy pull-back of the diagram one gets by
forgetting the initial object $A$. We compare the difference
between $A$ and this homotopy pull-back.
This difference is measured in terms of the homotopy fibers of the original maps.
Restricting our attention to the connectivity of these maps, we recover the classical
Blakers-Massey Theorem.
\end{abstract}


\maketitle


\section*{Introduction}
\label{sec intro}
The way  spaces are often studied in homotopy theory  is by decomposing and approximating them using
simpler and possibly better understood pieces. This is typically done in two ways. One way is via cellular 
approximations, where basic building blocks are assembled together using homotopy push-outs.    The other 
way is by glueing basic building blocks using homotopy pull-backs to form Postnikov or Bousfield-Kan type 
completion towers. These two approaches let us focus and extract different types of information about a given 
space.  It is not uncommon that one has some knowledge about    one approximation and needs to understand 
properties of the space detected by the other approximation. For example in~\cite{CDFS2014}, the objective is 
to  describe cellularity properties of Postnikov  sections and spaces in Farjoun's modified Bousfield-Kan tower. 
This is done via a generalization of the ``Bousfield Key Lemma'', \cite{B2}. In this process it is essential to understand 
and quantify  differences between  homotopy push-outs and pull-backs.  Estimating the difference just in connectivity, 
as is given in the classical Blakers-Massey theorem~\cite{MR0044836} and its generalizations by  Brown and 
Loday, \cite{MR872255}, Ellis and Steiner \cite{MR897011},
and  Goodwillie in \cite{MR93i:55015} (see also the recent treatment in \cite{cubical} by Munson and Voli\'{c}), 
is  not enough for these purposes. One needs to quantify this difference in terms 
of so called cellular and acyclic relations: recall~\cite{MR1408539,Farjoun95} that a space $A$ {\it kills\/} $X$ (which we denote
by $X>A$) if  $X$ can be  constructed out of $A$ by means of  homotopy push-outs, telescopes, and extensions by fibrations.
Our main result here  is the following ``inequality'' which plays an essential role in~\cite{CDFS2014}.
\medskip

\begin{thm}\label{main theoremBM}
Consider a commutative square:
\[\xymatrix@C=1.5pc@R=1.5pc{
 A \ar[r]^{f} \ar[d]_{g} & B \ar[d]\\
 C \ar[r] & D}
\]
with total fiber~$\mathcal T$  and push-out fiber $\mathcal R$ (see~\ref{pt totalfiber}). Then
$\mathcal T>\OmegaB \Fibset(f)\ast\OmegaB \Fibset(g) \cup \OmegaB \mathcal R$.
\end{thm}

Note that in the above  theorem we make no connectivity assumptions, neither do we mention any choice of base points. 
We therefore need to explain what we mean by homotopy fibers, total fibers, and loop spaces  in situations which are not covered by   
standard  conventions that involve connectivity and choices of base points. This is why we use upright font for the loop space
and calligraphic letters for possible \emph{sets} of fibers. That is the content of Subsections~\ref{fibers} 
and~\ref{convention empty}. One might wonder if it is justifiable to introduce these constructions  just to make the 
theorem assumptions free. In our opinion it is. It is not uncommon that one performs operations resulting in non 
connected spaces and under which it is not easy to keep track of base points, for example taking fibers, loops, and pull-backs. 
To measure to what extent such an operation  converts a homotopy push-out into a homotopy pull-back, which lies at the heart of 
Goodwillie calculus for example,  one is then forced to study   squares of non-connected spaces without any particular 
choice of base points. Squares of non connected spaces occur  also typical in    inductive arguments when one studies higher  
dimensional cubes.  In many of these
cases having Blakers-Massey type of statements without constraints and assumptions is very useful.

We also need to explain what we mean by $\mathcal T$ being killed by a \emph{union} of spaces, which we do in Subsection~\ref{point acyclic set}.
To give the reader an idea of the meaning of our main result let us state it in a situation where the terms and symbols refer
to the standard notions. In order for the homotopy fibers to be well-defined and connected we assume that all spaces are
simply connected and so as not to have to deal with the push-out fiber we assume furthermore that the square is a homotopy
push-out.

\medskip

\noindent
{\bf Corollary~\ref{cor classicalBM}.}
\emph{Consider a homotopy push-out square of pointed and simply connected spaces:
\[\xymatrix@C=1.5pc@R=1.5pc{
 A \ar[r]^{f} \ar[d]_{g} & B \ar[d]\\
 C \ar[r] & D}
\]
with total fiber~$T$. Then
$T>\Omega \Fib(f)\ast\Omega\Fib(g)$.}

\medskip

This implies in particular that the connectivity of $T$ is bounded below by the connectivity of the join $\Omega \Fib(f)\ast\Omega\Fib(g)$, i.e.~the
sum of the connectivities of the homotopy fibers $\Fib(f)$ and $\Fib(g)$. This gives  
Corollary~\ref{thm BMclassical}, the classical Triad Theorem of Blakers and Massey, \cite[Theorem~I]{MR0044836}.

\medskip

\noindent 
{\bf Acknowledgments.} We would like to thank the  referee for their comments and sensible
suggestions concerning the presentation of the proof. They lead us in particular to pay careful
attention to statements concerning sets of fibers rather than a single one.

\section{Notation and connectivity issues}
\label{sec notation}
In this first section we set up the notation for basic constructions and acyclic classes. 
We also  carefully introduce our conventions related to connectivity.%

\begin{point}{\bf Basic homotopical constructions.}
We work in the category of simplicial sets, so a space means a simplicial set. 
As usual, all the results presented in this paper can be translated to topological spaces
by using
any  Quillen equivalence between the model categories of simplicial sets and topological spaces
with weak homotopy equivalences.

Let $X$ be a space. For any $x$ in $\pi_0X$, the symbol $X_x$ denotes the connected component of $X$  corresponding to $x$.

The homotopy push-out  of $\Delta[0]\leftarrow A\to\Delta[0]$ is called the  suspension of $A$ and is denoted by $\Sigma A$.  
The symbol $S^{-1}$  denotes the empty simplicial set. For $n\geq 0$, the space $\Sigma^{n+1} (S^{-1})$ is also denoted by $S^{n}$ 
and called the $n$-sphere. For example $S^0=\Sigma S^{-1}$ is weakly equivalent to $\Delta[0]\coprod \Delta[0]$.  
More generally, for $n\geq 0$,  $S^n$ is weakly equivalent to the boundary of the simplex $\Delta[n+1]$.

Any contractible space containing  $A$ is called a cone over $A$ and denoted by $CA$.

Let $X$ and $Y$ be  spaces. The homotopy push-out of the following diagram is called the join of $X$ and $Y$ and is denoted by $X\ast Y$:
\[\xymatrix{
X &X\times Y\ar[l]_-{\text{pr}_X}\ar[r]^-{\text{pr}_Y} & Y
}\]
Note  that $S^{-1}$ is a unit for the join construction since $ \hocolim(S^{-1} \leftarrow 
S^{-1}\times X \to X)$ is weakly equivalent to $X$.  Furthermore the join $S^0\ast X$ is weakly equivalent to the suspension $\Sigma X$. 

A choice of a base point $x\colon\Delta[0]\to X$ in $X$, leads to an inclusion $Y\hookrightarrow X\times Y$ whose
quotient is denoted by
$X\rtimes Y$.  If  we also choose a base point $y\colon\Delta[0]\to Y$ in $Y$, we get an inclusion
$X\vee Y\subset X\times Y$ of the wedge into the product whose quotient is denoted by $X\wedge Y$ and called the smash product of $X$ and $Y$. 
Since for connected spaces $X$ and $Y$, the homotopy types of $X\rtimes Y$, the wedge $X\vee Y$, and  the smash product $X\wedge Y$ 
do not depend on the choice of the base points in $X$ and $Y$, we will use these constructions  for connected spaces without mentioning any base points.

If  $X$ and $Y$ are non-empty, then for any choice of base points in $X$ and $Y$, the join $X\ast Y$ is weakly equivalent to the suspension of the smash 
$\Sigma (X\wedge Y)$.

Let $f\colon A\to X$ be a map of spaces. We use the symbol $X/A$ to denote any space that fits into a homotopy push-out square
\[\xymatrix@C=1.5pc@R=1.5pc{A\ar@{->}[d]\ar[r]^-{f} & X\ar[d]\\
C\ar[r] &X/A}\]
where $C$ is contractible. More precisely we  first  require the above square to be strictly commutative and second, for some (equivalently any) factorization 
$A \hookrightarrow C' \xrightarrow{\simeq} C$  of the map $A\to C$ into  a cofibration followed by  a weak equivalence, we require the  induced   comparison  map from 
the push-out $C' \cup_A X$ to $X/A$ to be a weak equivalence. 

We call the space $X/A$  the homotopy cofiber of $f$, the map $X\to X/A$ the ``homotopy cofiber map'',
and the sequence of maps $A\to X\to X/A$ a cofiber sequence.
Two different choices of homotopy cofibers are always linked by a zigzag
of weak equivalences under $A$, so that the space $X/A$ is well defined up to weak equivalence and the homotopy cofiber map is
well defined up to weak equivalence under $A$.
\end{point}

\begin{point}{\bf Sets of spaces.}\label{setsofspaces}
Two sets of spaces $M$ and $N$ are equivalent if, for any space $X$ in $M$, there is a space $Y$ in $N$ which is weakly equivalent to $X$ and vice-versa, for any space $Y$ in $N$, there is a space $X$ in $M$ which is weakly equivalent to $Y$.

We  frequently apply constructions (e.g.\  the suspension,   the join) to sets of spaces 
and it is always understood that these constructions should be applied to their elements. So for example, 
for a set of spaces $M$, its suspension is $\Sigma M=\{ \Sigma X\ |\ X\in M\}$. Similarly, the join of two sets of spaces $M$ and~$N$ is given by
$M\ast N = \set{X\ast Y}{X\in M,\, Y\in N}$. 
\end{point}

\begin{point}{\bf Fibers.}\label{fibers}
Let $f\colon A \rightarrow B$  be a map  and $b$ in $\pi_0B$ be a component. The symbol 
$\Fib_b(f)$  denotes any space that fits into a homotopy pull-back square
\[\xymatrix@C=1.5pc@R=1.5pc{
\Fib_b(f)\ar[r]\ar[d] & A\ar[d]^{f}\\
P \ar[r] & B}
\]
where $P$ is a contractible space and the image of the map $P \to B$ lies in the  component~$B_b$.
More precisely we first  require the above square to be strictly commutative and second, for some (equivalently any)  factorization $P \xrightarrow{\simeq} P' \twoheadrightarrow B$ of the map  $P\to B$ into a weak equivalence 
 followed by a fibration,  we require the  induced  comparison map from $\Fib_b(f)$ to the pullback $P' \times_B A$ to be
a weak equivalence. 
We call the space  $\Fib_b(f)$ the homotopy fiber of $f$ over $b$, the map $\Fib_b(f)\to A$ the ``homotopy fiber map'', 
and the sequence $\Fib_b(f)\to A\to B$ a fibration sequence. Two different choices of homotopy fibers are always linked by a zigzag
of weak equivalences over $A$, so that the space $\Fib_b(f)$ is well defined up to weak equivalence and the homotopy fiber map is
well defined up to weak equivalence over $A$.

We use the symbol $\Fibset(f)$ to denote the set of spaces  $\{ \Fib_b(f) \, | \, b \in \pi_0(B) \}$ and call it the \emph{homotopy fiber set}
of $f$. Note that $\Fibset(f)$ is  the empty collection if and only if $f$ is the map $\text{id}_{S^{-1}}\colon S^{-1}\to S^{-1}$. 
If the set $\Fibset(f)$ is  equivalent to a set   containing only one space, than  we  use the  symbol $\Fib(f)$ to denote that space and call it the homotopy fiber of $f$. In particular whenever  the symbol $\Fib(f)$ is used, it is assumed that $\Fibset(f)$ is  equivalent to a set   containing only one space. This is the case  when for example $B$ is connected.

Consider a homotopy pull-back square:
\[\xymatrix@C=1.5pc@R=1.5pc{
 A \ar[r]^f \ar[d] & B \ar[d]^h\\
 C \ar[r]^k & D 
 }
\]
In general, the sets of spaces $\Fibset(f)$ and $\Fibset(k)$ may not be equivalent. For them to be equivalent
the assumption that $\pi_0h\colon\pi_0B\to\pi_0D$ be an epimorphism is needed.

\begin{definition}\label{def closure}
A collection of spaces $\Ccal$ is called \emph{closed under extensions by fibrations}  if, for any map  $f\colon X\to Y$ such that 
$Y$ is in $ \Ccal$ and  $\Fibset(f)\subset\Ccal$, the space  $X$ is also a member of $\Ccal$.
\end{definition}
\end{point}

\begin{point}\label{point acyclic}
{\bf Acyclic classes.}
Every  space $A$ determines a nullification or periodization functor $P_A$, \cite{B2}, \cite{Farjoun95}.
The class $\overline{\Ccal(A)}$ consists of those spaces $X$ for which  $P_A X$ 
is contractible ($P_{A}$ ``kills'' $X$). That is to say, $\overline{\Ccal(A)}$ is the class of spaces that become contractible after localizing the 
category of spaces at $\{A\to\Delta[0]\}$. 
The relation  $X\in \overline{\Ccal(A)}$ is also denoted by $X>A$ and called an acyclic inequality. 
If $X>A$, then we  say that $X$ is killed by $A$. Note that a retract of a space in $\overline{\Ccal(A)}$  belongs to 
$\overline{\Ccal(A)}$. Furthermore, $\overline{\Ccal(A)}$  is closed under extensions by fibrations (see Definition~\ref{def closure}).

If $A$ is a non-empty space, then  $\overline{\Ccal(A)}$ is in fact   the smallest class  containing $A$ and closed under
weak equivalences, pointed homotopy colimits, and extensions by fibrations, \cite[Theorem~17.3]{MR1408539}.

\begin{example}\label{ex connectivitiy}
 \begin{enumerate}
   \item $\overline{\Ccal(S^{-1})}$ is the class of all spaces;
   \item $\overline{\Ccal(S^0)}$ is the class of all non-empty spaces. More generally, if $|\pi_0 A|>1$ ($A$ has more than one component),
   then $\overline{\Ccal(A)}$ is the class of all non-empty spaces;
   \item For $n\geq 0$, $\overline{\Ccal(S^{n+1})}$ is the class of  all spaces $X$ for which $|\pi_0 X|=1$ (it has exactly one component and in particular is not  empty) and $\pi_i(X)=0$ for $n\geq i>0$. Such spaces are called $n$-connected. Spaces which are $0$-connected are  called connected.
     \end{enumerate}
\end{example}
The above examples justify to call a  space $(-1)$-connected if it is not empty, whereas every space is $(-2)$-connected, which is exactly the convention used in 
\cite{MR93i:55015}. Note that this  is consistent with the connectivity of $X\ast Y$ being the sum of the connectivities of~$X$ 
and~$Y$ plus~$2$ (cf. \cite[Corollary 4.8 (2)]{MR1452856}). 
\end{point}

\begin{point}\label{point acyclic set}
{\bf Acyclic classes for sets of spaces.}
Given a  set~$M$ of spaces, the symbol $\overline{\Ccal(M)}$ denotes the class 
of spaces that become contractible 
after localizing   at   $\{X\to\Delta[0]\ |\ X\text{ in } M\}$.  If $M$ is the empty set, then 
$\overline{\Ccal(M)}$ consists of all contractible spaces. If $M$    contains the empty space, then
$\overline{\Ccal(M)}$ is the collection of all  spaces $\overline{\Ccal(S^{-1})}$.  If  $M$  consists of non-empty spaces,
choose one base point in each space  $X$ in $M$, and take
the wedge $A \defas \bigvee_{X\in M} X$; then  $\overline{\Ccal(M)} = \overline{\Ccal(A)}$. 
In particular, $\overline{\Ccal(A)}$ does not depend on the choices of base points.

Given two  sets of spaces $M$ and $N$, we write $M>N$ and say that~$N$ \emph{kills}~$M$ if $\overline{\Ccal(M)} 
\subseteq \overline{\Ccal(N)}$ (i.e.~$M \subseteq \overline{\Ccal(N)}$).  In particular any set of spaces  kills  the empty collection.
For a space $X$, the relation  $M>\{X\}$ is also denoted by  $M>X$. For example if $M>S^n$ (for $\geq -1$), then we say that
$M$ is $(n-2)$-connected.

The extension by fibration property of the acyclic inequality can be  expressed as:
\begin{lemma}
\label{lemma closure}
\begin{enumerate}
\item $E > \Fibset(p)\cup \{B\}$ for any map   $p\colon E \rightarrow B$.
\item   $\Fibset(gf)>\Fibset(f)\cup \Fibset(g)$  for any  composable maps $f\colon X\to Y$ and   $g\colon Y\to Z$.
\end{enumerate}
\end{lemma}

\end{point}

\begin{point}{\bf Loops.}\label{convention empty}
Let $X$ be a space. We are going to use the following notation:
\[{ \OmegaB} X\defas\begin{cases}
\text{Fib}(\Delta[0]\to X) & \text{ if $X$ is connected}\\
S^{-1} & \text{ if $X$ is not connected}
\end{cases}\]
For a connected and pointed space $X$, the space ${\OmegaB} X$
is weakly equivalent to  the standard   loop space $\Omega X$. 
We should remark however, that with this convention, some caution is  required when dealing with maps and the suspension loop adjunction. 
For example, we cannot apply this construction $\OmegaB$ to an arbitrary map. 
We can do that, however, if the range of the map is either connected or empty. 

The reason for this definition, other than making our main theorem work, is that looping should lower the connectivity. 

Let $M$ be a set of spaces. The symbol $\OmegaB M$ denotes the set of spaces of the form $\OmegaB X$ for any $X$ in $M$. 
Note that with our convention it is still true that  $M>N$ implies $\OmegaB M >\OmegaB N$~\cite[Theorem~3.4~(4)]{MR1464865}. 

Let $X$ be a  space and $CX$ its cone. Consider the homotopy push-out square
\[
\xymatrix@C=1.5pc@R=1.5pc{
X \ar[r] \ar[d] & CX \ar[d]\\
CX \ar[r]  & \Sigma X}
\]
By removing the initial space in the above square and taking the homotopy pull-back of the remaining diagram we obtain a space weakly 
equivalent to $\OmegaB \Sigma X$. The map $\eta\colon X \to \OmegaB\Sigma X$ induced by commutativity of  the above diagram is called
the canonical comparison or James map.  
 This map has the following property. If $Y$ is connected and pointed, then for any 
morphism $f\colon X\to \Omega Y$ in $\text{Ho}(\text{Spaces}_*)$, there is a unique morphism $f^\sharp\colon \Sigma X\to Y$
in $\text{Ho}(\text{Spaces}_*)$ for which  $f=(\Omega f^\sharp) \eta$. This is the standard suspension loop adjunction 
and can be proved either by an elementary direct calculation
-- using explicit models for cones, cylinders, and homotopies -- 
or by universal properties in a derivatoresque framework, see \cite[Proposition~3.17]{MR3031644} 
for inspiration. The map  $f^\sharp$ is called the adjunct of $f $.
\begin{lemma}
\label{lemma loop-suspension}
For any $X$, $\Fibset(\eta\colon X \to \OmegaB\Sigma X)>\OmegaB X\ast \OmegaB X$.
\end{lemma}

\begin{proof}
The case  $X$ is connected has been proved in \cite[Theorem~7.2]{MR1464865}.
With our convention that $\OmegaB X=S^{-1}$ if $X$ is not connected, the lemma is vacuously true in that case as well.
\end{proof}
\end{point}

\begin{point}{\bf Total fibers and push-out fibers.}\label{pt totalfiber}
Given a commutative square
\[
\xymatrix@C=1.5pc@R=1.5pc{
 A \ar[r] \ar[d] & B \ar[d]\\
 C \ar[r] & D & *!<2em,.7ex>{,} 
 }
\]
there is a natural map from $A$ to the homotopy pull-back $P$ of  $B \rightarrow D
\leftarrow C$ (the diagram obtained from the square after removing the initial object $A$).  
The homotopy fiber set  $\Fibset(A\to P)$ is called the \emph{total fiber set} of the square above and is often denoted by
$\mathcal T$. If $\mathcal T$ is equivalent to a set with only one space (for example if $P$
is connected),  then  this space is called the \emph{total fiber}  of the square and is often denoted by the symbol~$T$.

If $B$   and $D$ are connected, then, alternatively, the total fiber can be obtained 
as a homotopy fiber between homotopy fibers. We extend Goodwillie's result, see \cite[Section~1]{MR93i:55015},
 to the case of a set of total fibers.

\begin{lemma}
\label{lemma total}
If $B$   and $D$ are connected, then the total fiber set is equivalent to the  homotopy fiber set  of the induced map
$\Fib(A \rightarrow B) \rightarrow \Fib(C \rightarrow D)$.
\end{lemma}

\begin{proof}
Under these assumptions the  homotopy fiber sets $\Fibset(A \rightarrow B)$ and $\Fibset(C \rightarrow D)$
are equivalent to sets with only one space, which  are denoted respectively by $\Fib(A \rightarrow B)$ and $\Fib(C \rightarrow D)$. However, as neither $\Fib(C \rightarrow D)$   nor the homotopy pull-back $P$
are necessarily connected, we need to deal  with homotopy fiber sets.

Observe that $\Fib(P \rightarrow B)$ and $\Fib(C \rightarrow D)$ are weakly equivalent and the induced square:
\[
\xymatrix@C=1.5pc@R=1.5pc{
 \Fib(A \rightarrow B) \ar[r] \ar[d] & A\ar[d]\\
 \Fib(C \rightarrow D) \ar[r] & P 
 }
\]
is a homotopy pull-back. The map $\Fib(C \rightarrow D) \rightarrow  P$ induces a surjection on the sets of components
and every choice of a  base point in $\Fib(C \rightarrow D)$ determines a base point in $P$. 
In the above homotopy pull-back square, the vertical homotopy fibers over 
these base points are  weakly equivalent, which proves the lemma. 
\end{proof}

There is also a natural map from the homotopy push-out $Q$ of $C\leftarrow A\to B$ (the diagram obtained from the square after removing 
the vertex $D$) into $D$. 
The homotopy fiber set  $\Fibset(Q\to D)$ is called 
the \emph{push-out fiber set} of the square, often denoted by~$\mathcal R$.
If $\mathcal R$ is equivalent to a set with only one space (for example if $D$
is connected),  then  this space is called the \emph{push-out fiber}  of the square and is often denoted by the symbol~$R$.

The total fiber set measures how far a square is from being a homotopy pull-back, while the push-out fiber set measures 
how far the square is from being a homotopy push-out.
\end{point}

\section{The main theorem and examples.}
\label{sec main}
The most important case we will deal with is that of a homotopy push-out square
\begin{equation}\label{maintheorem}\vcenter{\hbox{$\xymatrix@C=1.5pc@R=1.5pc{
 A \ar[r]^{f} \ar[d]_{g} & B \ar[d]^{h}\\
 C \ar[r]^{k} & D & *!<2em,.5ex>{.}}$}}
\end{equation}
We will refer to this diagram and use the same names for the spaces and maps of such a push-out square
throughout the whole article.

\begin{theorem}
\label{thm BMsquare}
Let  $\mathcal T$ be the total fiber set of a homotopy push-out square~(\ref{maintheorem}). Then ${\mathcal T} >
\OmegaB\Fibset(f)\ast\OmegaB\Fibset(g)$.
\end{theorem}

The proof is given in Section~\ref{sec maintheorem}. In order to relate this statement to a more conventional situation,
let us state the same result for simply connected and pointed spaces, just like we did in the introduction.

\begin{corollary}\label{cor classicalBM}
Consider a homotopy push-out square of pointed and simply connected spaces:
\[\xymatrix@C=1.5pc@R=1.5pc{
 A \ar[r]^{f} \ar[d]_{g} & B \ar[d]\\
 C \ar[r] & D}
\]
with total fiber~$T$. Then
$T>\Omega \Fib(f)\ast\Omega\Fib(g)$.
\end{corollary}

When we only pay attention to the connectivity of the fibers, we
obtain, as a straightforward corollary, the classical triad theorem of
Blakers-Massey, \cite[Theorem~I]{MR0044836}, or rather its reformulation by
Goodwillie in \cite[Theorem~2.3]{MR93i:55015}.

\begin{corollary}
\label{thm BMclassical}
If, in the homotopy push-out square~(\ref{maintheorem})  above,
$ \Fibset(f)$ is $n$-connected and\/ $\Fibset(g)$ is $m$-connected for
some $m$,~$n\geq -1$, then the total fiber set\/ $\mathcal T$ of the square is
$(m+n)$-connected.
\end{corollary}

\begin{proof}
The connectivity assumptions can be reformulated as inequalities 
$ \Fibset(f)>S^{n+1}$ and $\Fibset(g)> S^{m+1}$. The claim now follows from 
$\OmegaB S^{n+1} * \OmegaB S^{m+1} > S^{n+m+1}$ and the fact that
the join construction preserves inequalities.
\end{proof}

In the rest of this section, we give examples illustrating various particular cases of  Theorem~\ref{thm BMsquare}. 

\begin{example}
\label{ex product}
Let $A$ and $B$ be connected spaces. Choose base points $a\colon\Delta[0]\to A$ and  $ b\colon\Delta[0]\to B$
and consider the induced collapse maps $A \leftarrow A \vee B \rightarrow B$. These maps fit into the following homotopy push-out square:
\[\xymatrix@C=1.5pc@R=1.5pc{
A \vee B \ar[r] \ar[d] & B \ar[d]\\
A \ar[r]  & \Delta[0]}
\]
This is a  typical  homotopy push-out with the terminal space being contractible. The total fiber $T$ of this  square
is the homotopy fiber of the inclusion $A \vee B \hookrightarrow A
\times B$. By Puppe's Theorem \cite[Lemma~2]{MR51:1808}, this is the join
$\Omega A \ast \Omega B$.  The same theorem gives weak equivalences
$\Fib(A \vee B \rightarrow A)\simeq B \rtimes \Omega A $ and 
$\Fib(A\vee B \rightarrow B)\simeq A \rtimes \Omega B$.
Thus in  this case, the total fiber of the square $\Omega A \ast \Omega B$ is a retract 
of the join
\[
\Omega(B \rtimes \Omega A) \ast 
\Omega(A \rtimes \Omega B) \simeq \Omega\Fib(A \vee B \rightarrow A)\ast \Omega\Fib(A\vee B \rightarrow B).
\]
This is much stronger than the inequality $T>\Omega\Fib(A \vee B \rightarrow A)\ast \Omega\Fib(A\vee B \rightarrow B)$
guaranteed by Theorem~\ref{thm BMsquare}.
\end{example}

\begin{example}
\label{ex which fibers}
In the previous example, the total
fiber of a homotopy push-out square was first expressed  using the homotopy fibers of
the horizontal and vertical maps to~$D$. This is not to be
expected in general as shown by the following example. Let us choose
an integral homology equivalence $X \rightarrow Y$, for example the
one described by Whitehead in \cite[Example~IV.7.3]{MR516508}, where $Y= S^1$ 
and $X$ is obtained from $S^1 \vee S^2$ by
attaching a single $3$-cell via the attaching map
\[
S^2 \xrightarrow{(2, -1)} S^2 \vee S^2 \hookrightarrow
\bigvee_{-\infty}^\infty S^2 \simeq \widetilde{S^1 \vee S^2}
\rightarrow S^1 \vee S^2.
\]
Here the first map has degree $2$ on the first sphere and degree
$-1$ on the second one, while the second map is the inclusion on the zeroth and first
factors of the infinite wedge. Finally the last map is the universal cover. We then consider the homotopy
push-out square
\[\xymatrix@C=1.5pc@R=1.5pc{
X \ar[r]^-f \ar[d]_-g & S^1 \ar[d]\\
\Delta[0] \ar[r]  & \Delta[0] & *!<2em,0ex>{.} }
\]
where $f$ is the first Postnikov  section.
The join of the loops of the homotopy fibers of $\Delta[0] \rightarrow \Delta[0]$
and $S^1 \rightarrow \Delta[0]$ is contractible, but
the total fiber $T$ is the universal cover $\tilde X$ of $X$, which is not contractible.

The inequality of our main theorem holds, however, since $\Fib(f)= \tilde X$ and $\Fib(g)=X$.
As $\Omega \Fib(g)$  is not connected,  $\overline{\Ccal(\Omega \Fib(f) * \Omega \Fib(g))}=\overline{\Ccal(\Sigma \Omega \tilde X)}$, 
and so  $T=\tilde X$ is killed by $\Omega \Fib(f) * \Omega \Fib(g)$ (see~\cite[Corollary~3.5 (2)]{MR1464865}).
\end{example}

Our last example illustrates both the necessity to deal with
non-connected spaces and the importance to consider all homotopy
fibers at once. It also confirms the usefulness of our convention about
$\OmegaB S^0 = S^{-1}$ to be able to deal with the borderline
cases.

\begin{example}
\label{ex disconected}
{\rm Let $f\colon S^1 \coprod S^1 \to  \Delta[0] \coprod S^1$ be the
disjoint union of the collapse  map and 
the identity map. Let  $g\colon S^1 \coprod S^1\to  S^1$  be the fold map (the identity on both copies
of $S^1$).   Consider the homotopy push-out diagram
\[\xymatrix@C=1.5pc@R=1.5pc{
S^1 \coprod S^1 \ar[r]^-{f} \ar[d]_-{g} & \Delta[0] \coprod S^1 \ar[d]\\
S^1 \ar[r]  & \Delta[0] }
\]
The homotopy pull-back $P$ is the disjoint union $(S^1
\times \Delta[2]) \coprod (S^1 \times S^1)$ and  the total fiber set $\mathcal T$ consists of a contractible space and 
$\Omega S^1$.

The  homotopy  fiber set  $\Fibset(f)$ in this example is equivalent to  $\{S^1, \Delta[0] \}$, whereas $\Fibset(g)$ is equivalent to 
$ \{S^0\}$.
Thus the join $\OmegaB \Fibset(f) * \OmegaB \Fibset(g)$  is equivalent to  the set $\{\Omega S^1, \Delta[0] \}$. Since $\{\Omega S^1, \Delta[0] \}>S^0$,
our Blakers-Massey Theorem tells us here that the total fiber set is killed by $S^0$, 
i.e.~every  space in the total fiber  set is non-empty.
}
\end{example}

\section{Reduction to connected $D$ and horizontal fibers.}
\label{sec horfibers}
The aim of this section is to prove the following  acyclic inequality between horizontal fibers in a homotopy push-out 
square. It  is a generalization of~\cite[Theorem~3.4 (1)]{MR1464865} to non-connected spaces.
The same argument used in the proof of this proposition will be used in the proof of Theorem~\ref{thm BMsquare}
to reduce it to the case where $D$ is a connected space. Thus, we see this section as an important step  along the way
to our main result, but the reader is invited to skip it on first reading if desired.

\begin{proposition}
\label{prop fibersinpushout}
For a homotopy push-out square (\ref{maintheorem}), we have  $\Fibset(k\colon C\to D) > \Fibset(f\colon A\to B)$.
In particular $\OmegaB \Sigma X > X$ for any space~$X$.
\end{proposition}

\begin{proof}
If $D$ is empty, then the whole square~(\ref{maintheorem}) consists of empty spaces and there is nothing to check.
Assume that $D$ is non-empty. For $d$ in $\pi_0D$ define $B_0\subset B$, $C_0\subset C$ and $A_0\subset  A$ 
to be the  preimages  of the connected component $D_d$ corresponding to $d$
along respectively  $h$, $k$ and $hf$. In this way we obtain a homotopy push-out square:
\begin{equation}\label{connectedbasicpush-out}\xymatrix@C=1.5pc@R=1.5pc{
 A_0 \ar[r]^{f_0} \ar[d]_{g_0} & B_0 \ar[d]^{h_0}\\
 C_0 \ar[r]^{k_d} & D_d}
\end{equation}
where the maps are the restrictions of  the corresponding maps from (\ref{maintheorem}).
The fibers of the maps in (\ref{maintheorem}) are the sums over $\pi_0D$ of the corresponding  
fibers of the square (\ref{connectedbasicpush-out}) and the same is true for total fibers. 
Therefore, the inequality we are looking for holds for 
the square (\ref{maintheorem}) if and only if, for
every connected component $d$ in $\pi_0D$, the same inequality holds for  the square 
(\ref{connectedbasicpush-out}). We can thus assume that $D$ is not only non-empty but also connected.

Connectedness of $D$  implies  $\pi_0f\colon\pi_0A\to \pi_0B$ is an epimorphism. If   $\pi_0f$ is not a bijection, then
one of the fibers in $\Fib(f)$ has more that one component. In this case the desired inequality 
holds as  any non-empty  space  is killed by a space with more than one component (see~\ref{ex connectivitiy}.(b)). We can thus assume that $\pi_0f\colon\pi_0A\to \pi_0B$ is a bijection. 
Next, think  about the maps $\pi_0f\colon\pi_0A\to \pi_0B$ and $\pi_0g\colon\pi_0A\to \pi_0C$ as functors between
discrete categories and form the Grothendieck construction (\cite[Section~38]{MR1879153}):
\[
I\defas\text{Gr}(\pi_0C\xleftarrow{\pi_0g}\pi_0A\xrightarrow{\pi_0f} \pi_0B)
\]
Define two functors $F,G\colon I\to \text{Spaces}$ as follows. On objects:
\[
F(i)\defas\begin{cases}
C_c &\text{ if }i=c\in\pi_0C\\
A_a &\text{ if }i=a\in\pi_0A\\
A_a &\text{ if } i=\pi_0f(a)\in\pi_0B
\end{cases}\ \ \ 
G(i)\defas\begin{cases}
C_c &\text{ if }i=c\in\pi_0C\\
A_a &\text{ if }i=a\in\pi_0A\\
B_b &\text{ if } i=b\in\pi_0B
\end{cases}
\]
On morphisms, the maps $F(a\to \pi_0g(a)), G(a\to \pi_0g(a))\colon  A_a\to C_{\pi_0g(a)}$  are equal and given by 
 the restriction of $g$, the map
$F(a\to \pi_0f(a))\colon  A_a\to A_{a}$  is set to be the identity, and 
$G(a\to \pi_0f(a))\colon  A_a\to B_{\pi_0f(a)}$ to  be the restriction of $f$. Note that the identity maps  and the restrictions of $f$ define a natural transformation $\phi\colon F\to G$. In this way we obtain two functors $F$ and $G$ 
whose values are connected, and a natural transformation $\phi\colon F\to G$ for which  $\Fib(\phi_i)$ is either
contractible or is a retract of a space in $\Fibset(f)$. These fibers belong therefore  to $\overline{\Ccal( \Fibset(f))}$ 
and we can then use~\cite[Theorem~9.1]{MR1397724} to conclude that $\Fibset(\text{hocolim}_I\phi)> \Fibset(f)$.
Since  the map $\text{hocolim}_I\phi$ is weakly equivalent to $k\colon C\to D$, we get  the inequality we aimed to prove.
\end{proof}

\section{Reduction to contractible $D$}
\label{sec reduction}
Homotopy push-out diagrams in which the terminal object is contractible
are easier to handle because the homotopy pull-back one needs to
form in order to compute the total fiber is simply a product. The
aim of this section is to reduce the Blakers-Massey theorem to this
situation, which Klein and Peter call a fake wedge in \cite{MR3192617}. 

\begin{proposition}
\label{prop reduction}
If in  the homotopy push-out square~(\ref{maintheorem})  the space $D$ is connected, 
then the spaces\/ $\Fib(hf)$, $\Fib(h)$, and\/ $\Fib(k)$ fit into a homotopy push-out square:
\begin{equation}\label{pointhopushput}\xymatrix@C=1.5pc@R=1.5pc{
\Fib(hf) \ar[r]^{f'} \ar[d]_{g'} & \Fib(h) \ar[d]\\
\Fib(k) \ar[r]  & \Delta[0]}
\end{equation}
with the following properties:
\begin{itemize}
\item{} $\Fibset(f')$ is equivalent to $\Fibset(f)$;
\item   $\Fibset(g')$ is equivalent to $\Fibset(g)$;
\item{} the total fiber set ${\mathcal T}'$ of~(\ref{pointhopushput}) is
equivalent to the total fiber set ${\mathcal T}$ of (\ref{maintheorem}).
\end{itemize}
\end{proposition}

\begin{proof}
Choose a fibration $P\twoheadrightarrow D$ with contractible $P$ and pull-back 
the square  (\ref{maintheorem}) along this map to form the following commutative cube:
\[\xymatrix@C=1.5pc@R=1.5pc{
&\Fib(h)\ar@{->>}[rr] \ar[dd]|\hole & & C\ar[dd]^-{h}\\
\Fib(hf)\ar@{->>}[rr] \ar[ur]^{f'}\ar[dd]_-{g'} & & A\ar[dd]_(.25){g}\ar[ur]^{f}\\
& P \ar@{->>}[rr]|(.55)\hole& & D \\
\Fib(k)\ar@{->>}[rr]\ar[ur] & & B\ar[ur]^-{k}
}\]
According to Mather's second Cube Theorem
\cite[Theorem~25]{MR53:6510}, the face in this cube containing $f'$ and $g'$ is a homotopy push-out square.
Since $P$ is contractible, the square:
\[\xymatrix@C=1.5pc@R=1.5pc{
\Fib(hf) \ar[r]^{f'} \ar[d]_{g'} & \Fib(h) \ar[d]\\
\Fib(k) \ar[r]  & \Delta[0]}
\]
is a homotopy push-out.  As the map $P\twoheadrightarrow D$ induces an epimorphism on $\pi_0$,
so do the maps $ \Fib(h)\twoheadrightarrow C$ and $ \Fib(k)\twoheadrightarrow B$.
This implies that the set $\Fibset(f)$ is equivalent to $\Fibset(f')$ and $\Fibset(g)$ is equivalent to $\Fibset(g')$.
Exactly the same argument gives an equivalence between the total fiber set of the former square with the total 
fiber set $\mathcal T$ of the square~(\ref{maintheorem}).
\end{proof}

\section{Cofibrations}
\label{sec cofibrations}
Consider a map $A \rightarrow X$ to a connected space $X$ and its homotopy fiber $F$. 
In this section, we give an
estimate for the total fiber set $\mathcal T$ of the following 
homotopy push-out square:
\begin{equation}\label{cofibration}
\xymatrix@C=1.5pc@R=1.5pc{
A \ar[r] \ar[d] & X \ar[d]\\
CA \ar[r]  & X/A}
\end{equation}

\begin{proposition}
\label{prop cof}
If  $X$ is connected, then 
${\mathcal T} > \bigl\{\OmegaB F * \OmegaB F,\, \OmegaB(F * \Omega X)\bigr\}$.
\end{proposition}

\begin{proof}
If~$F$ is not connected then $\OmegaB F\ast\OmegaB F $ is the empty space by our  
convention (see~\ref{convention empty}) and the claim is trivial. So we assume that~$F$ is connected.
This implies that $A$ is also connected, and then so is $X/A$. Hence a choice of a base point
in $F$ turns this situation into a pointed one.
The total fiber set $\mathcal T$ is then the homotopy fiber of the  induced map 
$\alpha\colon F \to \Omega(X/A)$, which factors through $\eta\colon F \to \Omega\Sigma F$ as $\alpha = (\Omega \alpha^\sharp)\circ \eta$, 
where $\alpha^\sharp\colon \Sigma F \to X/A$ is the adjunct of~$\alpha$ (see Section~\ref{convention empty}). 
Using~\ref{lemma closure}.(b), we then obtain
$\mathcal T=\Fibset(\alpha)=\Fibset((\Omega \alpha^\sharp)\circ \eta)>\Fibset(\Omega \alpha^\sharp)\cup \Fibset( \eta)$.

According to~\ref{lemma loop-suspension}, $ \Fib( \eta)>\Omega F\ast\Omega F.$
The  adjunct map $\alpha^\sharp$ fits into the 
following commutative diagram, where all the squares are homotopy push-outs:
\[
\xymatrix@C=1.5pc@R=1.5pc{
A \ar[r] \ar[d] & A/F \ar[r] \ar[d] & X \ar[d]\\
CA \ar[r]  & \Sigma F \ar[r]^-{\alpha^\sharp} & X/A & *!<2em,.5ex>{.}}
\] 
By Proposition~\ref{prop fibersinpushout}, 
$\Fib( \alpha^\sharp)>\Fib(A/F\to X)\simeq  F \ast \Omega X$ which yields
$ \Fib( \Omega\alpha^\sharp)>\Omega(F \ast \Omega X)$. These two relations give the desired inequality.
\end{proof}

\section{A  rough estimate}
\label{sec rough}
In this section, we obtain a first, rather  rough estimate for the total
fiber. By combining this seemingly  weak  estimate with our
results for cofibration sequences, we will be able to prove  Theorem~\ref{thm BMsquare} in Section~\ref{sec maintheorem}. 

Throughout this section let us    fix a homotopy push-out square of the form:
\[
\xymatrix@C=1.5pc@R=1.5pc{
A \ar[r]^{f} \ar[d]_{g} & B \ar[d]\\
C \ar[r]  & \Delta[0]}
\]
In the case $B$ is connected we use the symbol $F\to A$ to denote the homotopy fiber map of $f$
over the unique component of $B$ (see~\ref{fibers}). Similarily, if $C$ is connected we use the symbol $G\to A$ to denote the homotopy fiber map of $g$ over the unique component of $C$. 
 The total fiber set $\mathcal T$ of the square above is by definition
the homotopy fiber set of the map $(f,g)\colon A \rightarrow B \times C$. By Lemma~\ref{lemma total},
when $B$ is connected,  this total fiber set can be  alternatively described as the homotopy fiber set 
of the map $\alpha\colon F\to C$ which is the composite of the homotopy fiber map $F\to A$ and $g$.

\begin{lemma}
\label{lemma 2-connected}
If $B$ and $C$ are connected, then the homotopy cofiber $C/F$ of the
map  $\alpha\colon F\to  C$ is killed by $F \ast \Omega B$.
In particular,  $C/F$  is $2$-connected if $F$ is $1$-connected.
\end{lemma}

\begin{proof}
We have a homotopy push-out square
\[\xymatrix@C=1.5pc@R=1.5pc{
A/F \ar[r] \ar[d] & B \ar[d]\\
C/F \ar[r]  & \Delta[0] & *!<2.2em,.4ex>{.}}
\]
Therefore, we infer from Proposition~\ref{prop fibersinpushout} that $C/F$ is killed by $\Fib(A/F \rightarrow B) \simeq
F \ast\Omega B$. If $F$ is  $1$-connected and $B$ is connected, this join is
$2$-connected.
\end{proof}

Here is our ``rough estimate''. The roughness of this cellular
inequality comes from the fact that it only involves one of the fibers. As we
know from the classical version of the Blakers-Massey Theorem, the connectivity
of the total fiber should be related to the sum of the connectivities of \emph{both} fibers.

\begin{proposition}
\label{prop rough}
If~$B$ and~${\mathcal T}$ are connected, then ${\mathcal T} > \Sigma\OmegaB F$.
\end{proposition}

\begin{proof}
If  $\OmegaB F$ is not connected, then it is either empty or contains $S^0$ as a retract. 
In the first case,  $\Sigma \OmegaB F=S^0$ and ${\mathcal T}>\Sigma \OmegaB F$ is clear, 
as any space in $\mathcal T$ is connected and hence non-empty.  If $S^0$ is a retract of $\Omega F$, then $S^1$ is a retract of $\Sigma \Omega F$ 
and hence ${\mathcal T}>\Sigma \Omega F$
follows from the assumption that all the spaces in  ${\mathcal T}$  are connected (see Example~\ref{ex connectivitiy}.(c)).

Let us  assume that $\Omega F$ is connected and that, therefore, $F$ is $1$-connected. This implies that  $A$ is connected.
Moreover, according to Proposition~\ref{prop fibersinpushout}, $C>F$ and so $C$ is   $1$-connected. 
By Lemma~\ref{lemma 2-connected}, we also know  $C/F> S^3$. 
The total fiber set ${\mathcal T}$ here consists of a single space $T$, which is equivalent to the homotopy fiber of the map 
$\alpha\colon F \rightarrow C$ by Lemma~\ref{lemma total}.
This map fits into the following commutative diagram:
\[\xymatrix@C=.9pc@R=.9pc{
& & & &C\ar[ddddd]|(.36)\hole|(.52)\hole \ar[rrrrrd]^-{\text{id}}\\
F\ar[rrrrrd]^(.6){\text{id}} \ar[rrrru]^{\alpha}\ar[ddddd] \ar[rd]|-\beta  & & & & &&  & & & \Delta[0]\ar[ddddd]\\
& H\ar[rrruu]|(.3)\hole\ar[ldddd]\ar[rrrrrd]^(.3){\gamma}|(.74)\hole  & & & & F \ar[rrrru]\ar[rd]|-{\eta}\ar[ddddd]\\
& & & & &&  \Omega\Sigma F \ar[rrruu]\ar[ldddd]\\
\\
& & & &C/F\ar[rrrrrd]|(.245)\hole|(.36)\hole \\
CF\ar[rrrrrd]\ar[rrrru]& & & & &&  & & & \Sigma F \\
& & & & &CF\ar[rrrru] &&&&&,
}\]
where:
\begin{itemize}
\item  the faces $\begin{gathered}\xymatrix@C=1.5pc@R=1.5pc{F\ar[r]^{\alpha}\ar[d] &C\ar[d]\\ CF\ar[r] & C/F}\end{gathered}$ and
$\begin{gathered}\xymatrix@C=1.5pc@R=1.5pc{F\ar[r]\ar[d] &\Delta[0]\ar[d]\\ CF\ar[r] & \Sigma F}\end{gathered}$
are homotopy push-outs;
\item  the faces $\begin{gathered}\xymatrix@C=1.5pc@R=1.5pc{H\ar[r]\ar[d] &C\ar[d]\\ CF\ar[r] & C/F}\end{gathered}$ and
$\begin{gathered}\xymatrix@C=1.5pc@R=1.5pc{\Omega\Sigma F\ar[r]\ar[d] &\Delta[0]\ar[d]\\ CF\ar[r] & \Sigma F}\end{gathered}$
are homotopy pull-backs.
\end{itemize}
In this way, we expressed $\alpha\colon F \rightarrow C$ as a composition of $\beta\colon F\to H$ and
$H\to C$, which  gives (see~\ref{lemma closure}):
\[
T =\Fib(\alpha\colon F \rightarrow C)>\Fib(\beta\colon F\to H)\cup \Fib(H\to C).
\]
To prove the proposition, it is then   enough to show that both fiber sets $ \Fib(H\to C)$ and 
$\Fib(\beta\colon F\to H)$ are killed by  $\Sigma\Omega F$.  That is what we are going to do.

We start with $\Fib(H\to C)$. 
Note that we have the following  sequence of relations:
\[
\Fib(H\to C)\stackrel{(a)}{\simeq} \Omega (C/F)\stackrel{(b)}{>} \Omega (F \ast \Omega B)\stackrel{(c)}{>}\Omega\Sigma F
\stackrel{(d)}{>} F\stackrel{(e)}{>}\Sigma \Omega F,
\]
where the weak equivalence $(a)$ is a consequence of the fact that the relevant square is a homotopy pull-back;
 the inequality $(b)$ follows from Lemma~\ref{lemma 2-connected}; connectedness of $B$ gives $(c)$;
Proposition~\ref{prop fibersinpushout} gives $(d)$; and finally $(e)$ is a consequence of the fact that $F$ is connected (see for example~\cite[Corollary 3.5]{MR1464865}).

It remains to show that $\Fib(\beta\colon F\to H)>\Sigma \Omega F$. 
The space $H$ is  the homotopy fiber of the cofiber  map $C\to C/F$ and hence  $H>F$ (see Proposition~\ref{prop fibersinpushout}).   
Thus $H$ is also $1$-connected and consequently, $\Fib(\gamma\colon H\to \Omega\Sigma F)$
is a  connected space.  According to the diagram above, the composition of $\beta\colon F\to H$ and 
$\gamma\colon H\to  \Omega \Sigma F$ is the James map $\eta\colon F \to \Omega \Sigma F$. 
The fibers of these three maps therefore fit into a fibration sequence
\[
\Fib(\beta\colon F\to H) \to \Fib(\eta\colon F \to \Omega \Sigma F) \to
\Fib(\gamma\colon H \to \Omega \Sigma F).
\]
We just have argued that the base in this fibration is connected. As $F$ is $1$-connected, so is the total space in this fibration.  
We can therefore form a new fibration sequence
\[\Omega \Fib(\gamma\colon H \to \Omega \Sigma F)\to 
\Fib(\beta\colon F\to H) \to \Fib(\eta\colon F \to \Omega \Sigma F).
\]
By Lemma~\ref{lemma loop-suspension}, $\Fib(\eta\colon F \to \Omega \Sigma F)>\Omega F *
\Omega F \simeq \Sigma (\Omega F \wedge \Omega F)$. Since $\Omega F$ is connected,
$ \Sigma (\Omega F \wedge \Omega F)>\Sigma \Omega F$. These two inequalities give
an   estimation for the base of the above fibration sequence:
$\Fib(\eta\colon F \to \Omega \Sigma F)>\Sigma \Omega F$. The desired inequality would then follow,
once we show $\Omega \Fib(\gamma\colon H \to \Omega \Sigma F)>\Sigma \Omega F$. 

Note that $\Fib(\gamma\colon H \to \Omega \Sigma F)$ is the total fiber of  the homotopy push-out square
\[
\xymatrix@C=1.5pc@R=1.5pc{
C \ar[r]^{} \ar[d] & \Delta[0] \ar[d]\\
C/F \ar[r]  & \Sigma F.}
\]
By  Proposition~\ref{prop cof}, this fiber is killed by
$\bigl\{\Omega H * \Omega H, \Omega(H * \Omega(C/F))\bigr\}$. Recall that $H>F>S^2$ and 
$C/F > S^3$ (see Lemma~\ref{lemma 2-connected}).
These inequalities  imply
\[
 \Fib(\gamma\colon H \rightarrow \Omega \Sigma F) > \bigl\{\Omega F * \Omega F,\,
 \Omega(F* S^2)\bigr\} > \{\Sigma^2 \Omega F,\, \Omega \Sigma^3 F\}.
\]
Since $\Omega \Sigma^3 F>\Sigma^2 \Omega F$, we obtain $\Fib(\gamma\colon H \rightarrow \Omega \Sigma F)>\Sigma^2 \Omega F$.
By looping this inequality, we finally get $\Omega \Fib(\gamma\colon H \rightarrow \Omega \Sigma F)>\Omega \Sigma^2 \Omega F>\Sigma \Omega F$. 
\end{proof}

Compared to the classical Blakers-Massey Theorem, the previous result might seem too strong. 
This is because our claim at the beginning of his section --~that we would  use only one fiber~--
was not entirely honest. We have used the fiber~$G$ implicitly in assuming that~$B$ is connected 
(implying that so is~$G$), which allowed us to pick up a suspension for the inequality $\mathcal{T} > \Sigma\Omega F$.
For a non-connected~$B$, one can only establish $\mathcal{T} > \Omega \Fibset(f)$, as the following example shows.

\begin{example}
Let $n\geq 0$ and $x\colon\Delta[0]\to S^n$ be a base point. 
 Consider the following homotopy push-out square
 \[ 
 \xymatrix@C=1.5pc@R=1.5pc{ S^n\coprod\Delta[0] \ar[r] \ar[d] &\Delta[0]\coprod\Delta[0] \ar[d] \\ S^n \ar[r] & \Delta[0], } 
 \]
where the left vertical map is given by the identity on $S^n$ and $x$ on $\Delta[0]$ and the top horizontal map
is the coproduct of the unique maps into $\Delta[0]$. Thus the homotopy fiber set ${\mathcal F}$  of the top horizontal map
is equivalent to  $\{S^n,\Delta[0]\}$. The total fiber set $\mathcal{T}$ of  this square is however equivalent to $\{\Omega S^n, \Delta[0]\}$. 
Thus in this case it is  not true that $\mathcal{T}>\Sigma\Omega {\mathcal F}$, even though, for $n>2$, every total fiber in $\Tcal$ is connected.
\end{example}

\section{Connectivity of the total fiber}
\label{sec CI}
Before we proceed to the proof of the ``acyclic Blakers-Massey Theorem'', we first need to establish a relationship 
between the connectivity of the fibers of the maps in a homotopy push-out square and the connectivity of its total fiber 
in order to be able to use Proposition~\ref{prop rough}.

\begin{proposition}
\label{prop Tcon1}
Assume that in the homotopy push-out square~(\ref{maintheorem}) the spaces $B$, $C$,  $D$,  $F = \Fib(f)$ and $G = \Fib(g)$ are connected. 
Then the total fiber set~$\mathcal T$ of this square  consists of one space which is connected.
\end{proposition}

\begin{proof}
The connectivity assumptions imply that the homotopy pull-back of the diagram $C\rightarrow D\leftarrow B$ is connected and hence  
the total fiber set $\mathcal T$ consists of one space $T$.  
Using Proposition \ref{prop reduction}, we assume, without loss of generality, that $D$ is contractible. Now, the maps 
$f_*\colon \pi_1(A)\to\pi_1(B)$ and $g_*\colon \pi_1(A)\to\pi_1(C)$ are surjective by connectedness of~$F$ and~$G$. 
Using the long exact homotopy sequence for $T \to A\to B\times C$, we need to show that 
$(f_*,g_*)\colon \pi_1(A) \to \pi_1(B)\times\pi_1(C)$ is surjective. But $\pi_1(B)\ast_{\pi_1(A)}\pi_1(C) \cong 1$ by the 
Seifert-van Kampen Theorem and the claim follows from the following lemma.
\end{proof}

We include the proof of the following group theoretical result, which is  probably well-known.

\begin{lemma}
Given a push-out diagram in the category of groups
 \[ 
 \xymatrix@C=1.5pc@R=1.5pc{ G \ar@{->>}[r]^{\phi} \ar@{->>}[d]_{\psi} & H \ar[d] \\ K \ar[r] & 1 } 
 \]
with~$\phi$ and~$\psi$ surjective, the homomorphism $(\phi,\psi)\colon G \to H\times K$ is surjective, too.
\end{lemma}

\begin{proof}
Writing $M = \Ker\phi$, $N = \Ker\psi$ and identifying $H \cong G/M$, $K\cong G/N$, we can reformulate 
the hypothesis $H\ast_GK \cong G/(M\mathbin{\raisebox{.35ex}{$\bigtriangledown$}} N) \cong 1$ as 
$M\mathbin{\raisebox{.35ex}{$\bigtriangledown$}} N = G$, where $M\mathbin{\raisebox{.35ex}{$\bigtriangledown$}} N$ 
is the normal closure of~$M\cup N$ in $G$. But~$M$ and~$N$ are normal subgroups and so 
$G = M\mathbin{\raisebox{.35ex}{$\bigtriangledown$}} N = MN$. By the second isomorphism theorem then 
$G/M = (MN)/M \cong N/(M\cap N)$ and $G/N \cong M/(M\cap N)$, so that both $\phi\colon N \to G/M \cong H$ 
and $\psi\colon M \to G/N \cong K$ are surjective. Finally, this implies the surjectivity of 
$(\phi,\psi)\colon G \to H\times K$ because if $(h,k)\in H\times K$, we find $n\in N$, $m\in M$ such that 
$\phi(n) = h$, $\psi(m) = k$ and thus $(\phi,\psi)(nm) = (h,1)(1,k) = (h,k)$.
\end{proof}

\section{The proof of Theorem~\ref{thm BMsquare}}
\label{sec maintheorem}
This section is devoted to the proof of Theorem~\ref{thm BMsquare}. 
\setcounter{equation}{0}
Recall that we are investigating the total fiber of a homotopy push-out square:
\begin{equation}\vcenter{\hbox{$\xymatrix@C=1.5pc@R=1.5pc{
 A \ar[r]^{f} \ar[d]_{g} & B \ar[d]^{h}\\
 C \ar[r]^{k} & D }$}}
\end{equation}
The first part consists in reducing the proof to the easier situation when the homotopy fibers of $f$ and $g$ are connected.
Just as in Section~\ref{sec horfibers} we invite the reader to skip this point on first reading.

\begin{point}
{\bf Reduction to connected fibers.}
If $D$  is empty, then so are $A$, $B$, $C$ and  the statement of the theorem is trivially true.
Assume then that $D$ is non-empty. By using the same argument as in the  first part of the proof of 
Proposition~\ref{prop fibersinpushout}, we can reduce the proof of the theorem to the case where $D$ is connected. 
A further reduction can be then obtained by using Proposition~\ref{prop reduction}, which states that the general case for a
connected $D$ follows from the case $D=\Delta[0]$. Let us then make this assumption $D=\Delta[0]$.
This implies that both of the functions  $\pi_0f\colon\pi_0A\to \pi_0B$ and  $\pi_0g\colon\pi_0A\to \pi_0C$
are surjective.

If both sets  $\Fibset(f)$ and $\Fibset(g)$ contain a non-connected space, then according to our convention, the acyclic classes 
$\overline{\Ccal(\OmegaB \Fibset(f))}$ and $\overline{\Ccal(\OmegaB \Fibset(g))}$  consist of all spaces and 
hence so does $\overline{\Ccal(\OmegaB \Fibset(f)\ast \OmegaB \Fibset(g))}$. It is then clear that the total fiber set
belongs to this acyclic class: ${\mathcal T} > \OmegaB \Fibset(f)\ast \OmegaB \Fibset(g)$.

We can then assume that the set $\Fibset(f)$   consists of connected spaces. 
Assume further that $\Fibset(g)$ contains at least one non-connected space. This implies 
that $M>M\ast \OmegaB\Fibset(g)$ for any set of spaces $M$.  
Since all the spaces in $\Fibset(f)$  are connected, the function 
$\pi_0f$ is a bijection.  Consequently, as  an easy $H_0(-,{\mathbf Z})$ calculation shows, the space  $C$ has to be connected for $D$ to be connected. 
Thus, for any space $T_0$ in $\mathcal T$, there is a space $F_0$ in $\Fibset(f)$ that fits into a fibration sequence $T_0\to F_0\to C$.
This implies  that $\Fib(T_0\to F_0)$ is equivalent to $\Omega C$.   Recall  that by Proposition~\ref{prop fibersinpushout},  $C>\Fibset(f)$, which implies 
 $\Omega C>\Omega \Fibset(f)$.  From Lemma~\ref{lemma closure}, we then obtain:
\[
T_0> \{\Omega C, F_0\}>\Omega \Fibset(f)\cup \{F_0\}>\Omega \Fibset(f)>\Omega \Fibset(f)\ast \Omega \Fibset(g)
\]
As this happens for any $T_0$ in $\mathcal T$, we get the desired inequality $\mathcal T>\Omega \Fibset(f)\ast \Omega \Fibset(g)$.
\end{point}

\begin{point}
{\bf Connected fibers.}
The remaining case is when both $\Fibset(f)$  and $\Fibset(g)$ consists of connected spaces. This implies that both
$B$ and $C$ are connected which has several consequences. One is that  both $\Fibset(f)$  and $\Fibset(g)$ are equivalent to
sets containing only one space and as before we  use the symbol $F\to A$ to denote the homotopy fiber map of $f$ and
$G\to A$  to denote the homotopy fiber map of $g$.
A second consequence is that by  Proposition~\ref{prop Tcon1} the total fiber set $\mathcal T$ consists also of a single 
connected space~$T$. Lastly   $A$ has to be  connected too.
To get an estimate for $T$, which is the homotopy fiber of the map $\alpha\colon F\to C$ by Lemma~\ref{lemma total},  
we consider  the following commutative diagram (compare with the proof of Proposition~\ref{prop rough}),
\[\xymatrix@C=0.7pc@R=.7pc{
F\ar[rrr]^-{\alpha}\ar[ddd]\ar[dr]_{\beta} & & & C\ar[ddd]\\
& H\ar[urr]\ar[ddl]\\
\\
CF\ar[rrr] & & & C/F
}\]
where the outside square 
is a homotopy push-out square and the inside square is a homotopy pull-back square.
We  analyze $T$ as the homotopy fiber
of the composite $F\xrightarrow{\beta} H \rightarrow C$,  which gives:
\[T>\Fibset(\beta\colon F
\rightarrow H)\cup \Fibset(H\to C).\]

The homotopy fiber set $\Fibset(\beta)$ is the total fiber set of the outside homotopy push-out square.
Since  $C$ is connected,  $\Fibset(\beta) > \bigl\{\Omega T * \Omega T, \Omega(T * \Omega C)\bigr\}$ (see Proposition~\ref{prop cof}).
We can then use the rough estimates from
Proposition~\ref{prop rough} (with respect to both $F$ and $G$) and the fact that
$C>F$ to  obtain
\[
 \Fibset(\beta) > \bigl\{\Omega \Sigma \Omega F * \Omega \Sigma \Omega G, \Omega
 (\Sigma \Omega G * \Omega F)\bigr\} > \Omega F * \Omega G,
\]
where we used the fact that $\Omega \Sigma X > X$ for any space $X$ (see Lemma \ref{prop fibersinpushout}).

Since $\Fib(H \rightarrow C) = \Omega(C/F)$ and $C/F>\Omega(F * \Omega B) > \Omega(F * \Omega G) > \Omega F * \Omega G$ by Lemma~\ref{lemma 2-connected}, we can conclude $\Fib(H \rightarrow C)>\Omega F * \Omega G$.
\end{point}

\section{Proof of Theorem~\ref{main theoremBM}}
From the case of a homotopy push-out square, we easily deduce the statement for an arbitrary square.
Recall that $\mu\colon Q \rightarrow D$ denotes the comparison map between the homotopy push-out and
the terminal object of the commutative square.

If necessary by modifying relevant maps into cofibrations and fibrations, 
the square from Theorem~\ref{main theoremBM} can be fitted into the following commutative diagram
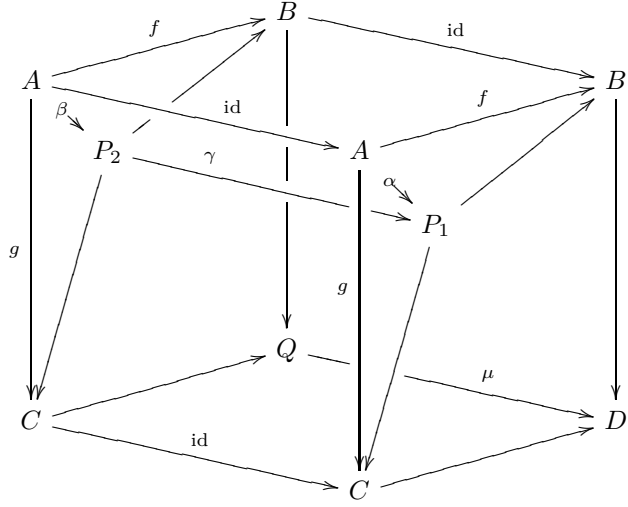
\[\xymatrix@C=.9pc@R=.9pc{
& & & &B\ar[ddddd]|(.365)\hole|(.53)\hole \ar[rrrrrd]^-{\text{id}}\\
A\ar[rrrrrd]^(.6){\text{id}} \ar[rrrru]^{f}\ar[ddddd]_{g} \ar[rd]|(.4){\beta}  & & & & &&  & & & B\ar[ddddd]\\
& P_2\ar[rrruu]|(.3)\hole\ar[ldddd]\ar[rrrrrd]^(.3){\gamma}|(.77)\hole  & & & & A \ar[rrrru]^{f}\ar[rd]|(.4){\alpha}\ar[ddddd]_(.4){g}\\
& & & & &&  P_1 \ar[rrruu]\ar[ldddd]\\
\\
& & & &Q\ar[rrrrrd]|(.225)\hole|(.325)\hole^(.6)\mu \\
C\ar[rrrrrd]^-{\text{id}}\ar[rrrru]& & & & &&  & & & D\\
& & & & &C\ar[rrrru] &&&&&,
}\]
where
\begin{itemize}
\item the face $\begin{gathered}\xymatrix@C=1.5pc@R=1.5pc{A\ar[r]^{f}\ar[d]_{g} &B\ar[d]\\ C\ar[r] & Q}\end{gathered}$ 
is a homotopy push-out square, while
\item the squares $\begin{gathered}\xymatrix@C=1.5pc@R=1.5pc{P_1\ar[r]\ar[d] &B\ar[d]\\ C\ar[r] &D}\end{gathered}$ and
$\begin{gathered}\xymatrix@C=1.5pc@R=1.5pc{P_2\ar[r]\ar[d] &B\ar[d]\\ C\ar[r] & Q}\end{gathered}$
are homotopy pull-back squares.
\end{itemize}
The total fiber set $\mathcal T$ of the square we are interested in is given by $\Fibset(\alpha\colon A\to P_1)$. 
The map $\alpha\colon A\to P_1$ is expressed as a  composition of  $\beta\colon A\to P_2$ and
$\gamma\colon P_2\to P_1$. We therefore get the inequality
\[
\mathcal T=\Fibset(\alpha\colon A\to P_1)>\Fibset(\beta\colon A\to P_2)\cup \Fibset(\gamma\colon P_2\to P_1).
\]
According to Theorem~\ref{thm BMsquare}, $\Fibset(\beta\colon A\to P_2)>\OmegaB \Fibset(f)\ast\OmegaB \Fibset(g)$.
The spaces in the fiber set $\Fibset(\gamma\colon P_2\to P_1)$ are among the spaces in $\OmegaB \Fibset(\mu\colon Q\to D)$
and since  $\Fibset(\mu\colon Q\to D)$ is the push-out fiber set $\mathcal R$ of the square in the theorem, we get 
$\Fibset(\gamma\colon P_2\to P_1)>\OmegaB \mathcal R$. These two inequalities give the inequality stated in the theorem.

\bibliographystyle{amsplain}
\providecommand{\bysame}{\leavevmode\hbox to3em{\hrulefill}\thinspace}
\providecommand{\MR}{\relax\ifhmode\unskip\space\fi MR }
\providecommand{\MRhref}[2]{%
 \href{http://www.ams.org/mathscinet-getitem?mr=#1}{#2}
}
\providecommand{\href}[2]{#2}



\end{document}